\documentclass[11pt]{article}
\usepackage{amsfonts}
\usepackage{mathrsfs}
\usepackage[dvips]{graphics}
\usepackage[dvips]{color}
\usepackage{amsthm}
\usepackage[]{amsmath}
\usepackage{CJK}
\usepackage{indentfirst}
\newtheorem{proposition}{Proposition}[section]
\newtheorem{lemma}[proposition]{Lemma}

\newtheorem{theorem}[proposition]{Theorem}
\newtheorem{corollary}[proposition]{Corollary}

\newtheorem{property}[proposition]{Property}
\newtheorem{remark}[proposition]{Remark}
\newtheorem{algorithm}[proposition]{Algorithm}
\begin{document}
\begin{CJK*}{GBK}{song}
\CJKindent

\centerline{\textbf{\LARGE{The structure of palindromes}}}

\vspace{0.2cm}

\centerline{\textbf{\LARGE{in the Fibonacci sequence
and some applications}}}

\vspace{0.2cm}

\centerline{Huang Yuke\footnote[1]{School of Mathematics and Systems Science, Beihang University (BUAA), Beijing, 100191, P. R. China.}$^,$\footnote[2]{E-mail address: huangyuke07@tsinghua.org.cn.}
~~Wen Zhiying\footnote[3]{Department of Mathematical Sciences, Tsinghua University, Beijing, 100084, P. R. China.}$^,$\footnote[4]{E-mail address: wenzy@tsinghua.edu.cn(Corresponding author).}}

\vspace{1cm}

\centerline{\textbf{\large{ABSTRACT}}}

\vspace{0.2cm}

Let ${\cal P}$ be the set of palindromes occurring in the Fibonacci sequence.
In this note, we establish three structures of $\mathcal{P}$ and and discuss their properties:
cylinder structure, chain structure and recursive structure. Using these structures, we determine that
the number of distinct palindrome occurrences in $\mathbb{F}[1,n]$ is exactly $n$, where $\mathbb{F}[1,n]$
is the prefix of the Fibonacci sequence of length $n$. Then we give an algorithm for counting the number of repeated palindrome occurrences in $\mathbb{F}[1,n]$, and get explicit expressions for some special $n$,
which include the known results. We also give simpler proofs of some classical properties, such as in X.Droubay\cite{D1995}, W.F.Chuan\cite{C1993} and J.Shallit et al\cite{DMSS2014}.

\vspace{0.2cm}

\noindent\textbf{Key words:} the Fibonacci sequence; palindrome; structure; algorithm;
the sequence of return words.

\section{Introduction}

As a classical example over a binary alphabet, the Fibonacci sequence, having many remarkable properties, appears in many aspects of mathematics and computer science
etc., we refer to M.Lothaire\cite{L1983,L2002}, J.M.Allouche and J.Shallit\cite{AS2003},
Berstel\cite{B1966,B1980}.

In this paper, we establish and discuss three structures of palindromes in the Fibonacci sequence.
These structures are important because palindromes are objects of a great interest in computer science, etc.
%De Luca\cite{D1981} gave the decompose of $F_m$ to the concatenation of two palindromes.
In 1993 and 1995, W.F.Chuan\cite{C1993} and X.Droubay\cite{D1995} determined the number of palindromes in conjugations of $F_m$ respectively, where $F_m=\sigma^m(a)$ and $\sigma$ is the Fibonacci morphism. X.Droubay\cite{D1995} also determined the number of palindromes of length $n$.
In 2014, C.F.Du, H.Mousavi, L.Schaeffer and J.Shallit\cite{DMSS2014} showed the necessary and sufficient condition of $\mathbb{F}[1,n]$ is a palindrome.
In this paper, 
we will give simper proofs of these properties above by the structures of palindromes, see Corollary 3.4, 3.5 and Remark 4.3.

On the other hand, as new results, we first determine the number of distinct palindrome occurrences in $\mathbb{F}[1,n]$ is exactly $n$ (Theorem 4.5), then we give an algorithm for counting
the number of repeat palindrome occurrences in $\mathbb{F}[1,n]$ (Algorithm 6.9), and give explicit expressions in some particular cases, such as when $n=f_m$ etc. (Theorem 6.7).

The main tool of this paper is the structure property of the sequence of return words in the Fibonacci sequence which introduced and studied in \cite{HW2015-1}.
By this property, we can determine the positions of all occurrences for each palindrome, and establish the sequence properties of them.
Moreover, we can count the number of palindromes occurrences in each prefix of $\mathbb{F}$, not only in $F_m$.

This paper is organized as follows.
Section 2 present some basic notations and known results. Sections 3 to 5 are devoted to establish three structure properties of palindromes in the Fibonacci sequence: cylinder sets, chain structure and recursive structure. In Sections 4 and 6, using the structure properties mentioned above, we determined also the number of distinct  and repeated palindrome occurrences in $\mathbb{F}[1,n]$ for all $n$ respectively.

\section{Preliminaries}

Let $\mathcal{A}=\{a,b\}$ be a binary alphabet.
A word is a finite string of elements in $\mathcal{A}$.
The set of all finite words on $\mathcal{A}$ is denoted by $\mathcal{A}^\ast$, which is a free monoid generated by $\mathcal{A}$ with the concatenation operation.
The concatenation of two words $\nu=x_1x_2\cdots x_r$ and $\omega=y_1y_2\cdots y_n$ is the word $x_1x_2\cdots x_ry_1y_2\cdots y_n$, denoted by $\nu\omega$. This operation is associative and has a unit element, the empty word $\varepsilon$.
The set $\mathcal{A}^\ast$ is thus endowed with the structure of a monoid, and is called the free monoid generated by $\mathcal{A}$.

Let $\mathcal{B}$ and $\mathcal{C}$ be two alphabets. A morphism is a map $\varphi$ from
$\mathcal{B}^\ast$ to $\mathcal{C}^\ast$ that the identity $\varphi(xy)=\varphi(x)\varphi(y)$ for all words $x,y\in\mathcal{B}^\ast$, see \cite{AS2003}.
The Fibonacci sequence $\mathbb{F}$ is the fixed point beginning with $a$ of the Fibonacci morphism $\sigma:\mathcal{A}^\ast\rightarrow \mathcal{A}^\ast$ defined over $\mathcal{A}^\ast$ by $\sigma(a)=ab$ and $\sigma(b)=a$.
$$\mathbb{F}=abaababa abaab abaababa abaababa abaab abaababa abaab abaababa\cdots$$
For the details of the properties of the sequence, see \cite{WW1994}.

Since $\sigma$ is a morphism, $\sigma(ab)=\sigma(a)\sigma(b)$.
The $m$-th iteration of $\sigma$ is $\sigma^m(a)=\sigma^{m-1}(\sigma(a))$ for $m\geq2$ and we denote $F_m=\sigma^m(a)$.
We define $\sigma^0(a)=a$ and $\sigma^{-1}(a)=\sigma^0(b)=b$.
Note that the length of $F_m=\sigma^m(a)$ is the $m$-th Fibonacci number $f_m$,
given by the recursion formulas $f_{-1}=1$, $f_0=1$, $f_{m+1}=f_m+f_{m-1}$ for $m\geq 0$.
Let $\delta_m\in\{a,b\}$ be the last letter of $F_m$. It's easy to see that
$\delta_m=a$ if $m$ is even; $\delta_m=b$ if $m$ is odd.

For a finite word $\omega=x_1x_2\cdots x_n$, the length of $\omega$ is equal to $n$ and denoted by $|\omega|$. We denote by $|\omega|_a$ (resp. $|\omega|_b$) the number of letters of $a$ (resp. $b$) occurring in $\omega$.
The $i$-th conjugation of $\omega$ is the word $C_i(\omega):=x_{i+1}\cdots x_nx_1\cdots x_i$ where $0\leq i\leq n-1$.
The mirror word $\overleftarrow{\omega}$ of $\omega$ is defined to be $\overleftarrow{\omega}=x_n\cdots x_2x_1$. A word $\omega$ is called a palindrome if $\omega=\overleftarrow{\omega}$.
Let $\mathcal{P}$ be all palindromes occurring in $\mathbb{F}$, and $\mathcal{P}(n)$ be all palindromes occurring in $\mathbb{F}$ of length $n$.

Let $\tau=x_1\cdots x_n$ be a finite word (or $\tau=x_1x_2\cdots$ be a sequence).
For any $i\leq j\leq n$, define $\tau[i,j]:=x_ix_{i+1}\cdots x_{j-1}x_j$. That means $\tau[i,j]$ is the factor of $\tau$ of length $j-i+1$, starting from the $i$-th letter and ending to the $j$-th letter. By convention, we denote $\tau[i]:=\tau[i,i]=x_i$ and $\tau[i,i-1]:=\varepsilon$.
%Word $\omega=\tau[i,j]$ is said to occur at position $i$ in $\tau$.
We say $\omega=\tau[i,j]$ is a factor of $\tau$, denoted by $\omega\prec\tau$.
When we say $\omega\prec\tau$, it can be that $\omega=\tau$.

We say that $\nu$ is a prefix (resp. suffix) of a word $\omega$, and write $\nu\triangleleft\omega$ (resp. $\nu\triangleright\omega$) if there exists $u\in\mathcal{A}^\ast$ such that $\omega=\nu u$ (resp. $\omega=u\nu$).
In this case, we can write $\nu^{-1}\omega=u$ (resp. $\omega\nu^{-1}=u$).

Let $A$, $B$, $C$ be three sets. We say $A$ is the disjoint unite of $B$ and $C$, if $A=B\cup C$ and $B\cap C=\emptyset$, denoted by $A=B\sqcup C$.

It is well known that the Fibonacci sequence $\mathbb{F}$ is uniformly recurrent, i.e., each factor $\omega$ occurs infinitely often and with bounded gaps between consecutive occurrences \cite{AS2003}.
We arrange them in the sequence
$\{\omega_p\}_{p\ge 1}$, where $\omega_p$ denote the $p$-th occurrence of $\omega$.

The definitions of return words and the sequence of return words below are from F.Durand\cite{D1998}. Let $\omega$ be a factor of $\mathbb{F}$. For $p\geq1$, let $\omega_p=x_{i+1}\cdots x_{i+n}$ and $\omega_{p+1}=x_{j+1}\cdots x_{j+n}$.
The factor $x_{i+1}\cdots x_{j}$ is called the $p$-th return word of $\omega$ and denoted by $r_{p}(\omega)$.
The sequence $\{r_p(\omega)\}_{p\geq1}$ is called the sequences of the return words of factor $\omega$.

The $m$-th singular word is defined as $K_m=\delta_{m+1}F_m\delta_m^{-1}$.
%, where $\delta_m\in\{a,b\}$ is the last letter of $F_m$.
It is known that all singular words are palindromes and $K_m=K_{m-2}K_{m-3}K_{m-2}$ for all $m\geq2$ see \cite{WW1994}.
Let $Ker(\omega)$ be the maximal singular word occurring in factor $\omega$, then by Theorem 1.9 in \cite{HW2015-1},
$Ker(\omega)$ occurs in $\omega$ only once.
Moreover

\begin{property}[Theorem 2.8 in \cite{HW2015-1}]\label{wp}~
$Ker(\omega_p)=Ker(\omega)_p$ for all $\omega\in\mathbb{F}$ and $p\geq1.$
\end{property}
This means, let $Ker(\omega)=K_m$, then the maximal singular word occurring in $\omega_p$ is just $K_{m,p}$. For instance, $Ker(aba)=b$, $(aba)_3=\mathbb{F}[6,8]$, $(b)_3=\mathbb{F}[7]$, so $Ker((aba)_3)=(b)_3$, $(aba)_3=a(b)_3a$.

\begin{property}[Theorem 2.11 in \cite{HW2015-1}]\label{G}~
For any factor $\omega$,
the sequence of return words $\{r_p(\omega)\}_{p\geq1}$ is the Fibonacci sequence over the alphabet $\{r_1(\omega),r_2(\omega)\}$.
\end{property}

%This is the main result in Huang and Wen\cite{HW2015-1}.

\section{The cylinder structure of palindromes}

We have known that all singular words are palindromes, in this section,
we will show that any palindrome can be generated by singular words.

\begin{property}[]\
Let $\omega\in\cal P$, $Ker(\omega)$ occurs in the middle of $\omega$.
\end{property}

\begin{proof}
Since both $\omega$ and  $Ker(\omega)$ are palindromes, $\overleftarrow{\omega}=\omega$ and $\overleftarrow{Ker(\omega)}=Ker(\omega)$. If $Ker(\omega)$ does not occur in the middle of $\omega$,
we can find two occurrences of $Ker(\omega)$ in $\omega$. That leads to a contradiction with
$Ker(\omega)$ occurs in $\omega$ only once.
%, and $Ker(\omega)$ as a singular word is a palindrome,
\end{proof}

By Corollary 2.10 in \cite{HW2015-1} and $K_{m+3}=K_{m+1}K_m K_{m+1}$, any factor $\omega$ with kernel $K_m$ can be expressed uniquely as
$$\omega=K_{m+1}[i,f_{m+1}] K_m K_{m+1}[1,j]=K_{m+3}[i,f_{m+2}+j],$$
where $2\leq i\leq f_{m+1}+1$, $0\leq j\leq f_{m+1}-1$. Thus by Property 3.1 we have

\begin{property}[]\
Any palindrome with kernel $K_m$ can be expressed uniquely as
$$K_{m+1}[i+1,f_{m+1}] K_m K_{m+1}[1,f_{m+1}-i]=K_{m+3}[i+1,f_{m+3}-i],\eqno(1)$$
where $1\leq i\leq f_{m+1}$.
\end{property}

This property above shows that for any $m\ge -1$, all palindromes with kernel $K_m$ can be listed as $\{K_{m+3}[i+1,f_{m+3}-i],1\leq i\leq f_{m+1}\}$, so any palindrome $\omega$ with kernel $K_m$, there exists
uniquely $i$ with $1\leq i\leq f_{m+1}$ so that $\omega=K_{m+3}[i+1,f_{m+3}-i]$.

By the expression (1), we see that the set $\cal P$ decomposes into three disjoint cylinder sets
$\langle a\rangle$, $\langle b\rangle$ and $\langle aa\rangle$, which can be illustrated by the following way, see Tab.1, which we call the cylinder structure of $\cal P$.
This means any palindromes can be generated by singular words.
More concretely, let $\omega$ be a palindrome, then $\omega\in \langle aa\rangle$ if $|\omega|$ is even;
$\omega\in\langle a\rangle$ (resp. $\langle b\rangle$) if $|\omega|$ is odd and the middle letter of $\omega$ is $a$ (resp. $b$).

\begin{remark}
Tab.1 shows the first several elements of cylinder sets $\langle a\rangle$, $\langle b\rangle$ and $\langle aa\rangle$, where we sign all singular words with underlines.
We can see that the singular words are sparse in $\mathcal{P}$, this means the palindromes are generated ``high efficient'' from singular words.
\end{remark}

\vspace{0.2cm}

%The first few values of
\centerline{Tab.1: The cylinder structure of palindromes.}
\vspace{-0.2cm}
\begin{center}
\begin{tabular}{ccc}
\hline
cylinder$\langle a\rangle$&cylinder$\langle b\rangle$&cylinder$\langle aa\rangle$\\\hline
\underline{a}&\underline{b}&\underline{aa}\\
\underline{bab}&aba&baab\\
ababa&\underline{aabaa}&abaaba\\
aababaa&baabaab&\underline{babaabab}\\
baababaab&abaabaaba&ababaababa\\
abaababaaba&babaabaabab&aababaababaa\\
\underline{aabaababaabaa}&ababaabaababa&baababaababaab\\
baabaababaabaab&aababaabaababaa&abaababaababaaba\\
abaabaababaabaaba&baababaabaababaab&aabaababaababaabaa\\
babaabaababaabaabab&abaababaabaababaaba&baabaababaababaabaab\\
ababaabaababaabaababa&\underline{babaababaabaababaabab}&abaabaababaababaabaaba\\
aababaabaababaabaababaa&ababaababaabaababaababa&babaabaababaababaabaabab\\
\vdots&\vdots&\vdots\\\hline
\end{tabular}
\end{center}

By the cylinder structure of palindromes, we can recover some results immediately.

\begin{corollary}[Droubay\cite{D1995}]~ For any $n\geq1$, $\mathcal{P}(n)=2$ if $n$ is odd; $\mathcal{P}(n)=1$ if $n$ is even.
\end{corollary}

\begin{corollary}[Chuan\cite{C1993} and Droubay\cite{D1995}]~ For any $m\geq-1$,
\begin{equation*}
\mathcal{P}(f_m)\cap\{C_i(F_m),0\leq i\leq f_{m}\}=
\begin{cases}
0,&m\equiv1~(mod~3)\\
1,&otherwise
\end{cases}
\end{equation*}
\end{corollary}

\begin{proof}
Notice that all factors with length $f_m$ are $\{K_m\}\sqcup\{C_i(F_m),0\leq i\leq f_{m}\}$ see \cite{WW1994}.
When $m\equiv 1$ (mod 3), $f_m$ is even, there is only one palindrome with length $f_m$, i.e. $K_m$, so there is no palindrome as a conjugation of $F_m$. When $m\equiv0,2$ (mod 3), $f_m$ is odd, there are exactly 2 palindromes with length $f_m$. One of them is $K_m$, the other one is a conjugation of $F_m$.
Thus the conclusion holds.
\end{proof}

%Corollary ?.? is the main result in Droubay\cite{D1995}.
%Chuan\cite{C1993} proved Corollary ?.? using many mathematical ad hoc elements, and Droubay\cite{D1995} proved it in another way.

\section{The chain structure of palindromes}

Recall that $\omega_p$ be the p-th occurrence of the factor $\omega$ in $\mathbb F$. We denote by $L(\omega,p)$ the position of the first letter of $\omega_p$, then the position of the last letter of $\omega_p$ is $P(\omega,p):=L(\omega,p)+|\omega|-1$.

Denote $\phi=\frac{\sqrt{5}-1}{2}$, $\lfloor\alpha\rfloor$ is the biggest integer not larger than $\alpha$.

\begin{property}\label{P}
For $m\geq-1$, $p\geq1$,
$P(K_m,p)=pf_{m+1}+(\lfloor\phi p\rfloor+1)f_{m}-1.$
\end{property}

\begin{proof}
By Proposition 3.4 in \cite{HW2015-1},
$L(K_m,p)=pf_{m+1}+|\mathbb{F}[1,p-1]|_af_{m}$. Since
$|\mathbb{F}[1,p-1]|_a=\lfloor\phi p\rfloor$
and $|K_m|=f_m$, we have $P(K_m,p)=L(K_m,p)+|K_m|-1=pf_{m+1}+(\lfloor\phi p\rfloor+1)f_{m}-1$.
\end{proof}

%Take $m=-1,0$ respectively,

\begin{corollary}\label{P1}
$P(a,p)=p+\lfloor\phi p\rfloor$,
$P(b,p)=2p+\lfloor\phi p\rfloor$
and $P(aa,p)=3p+2\lfloor\phi p\rfloor+1$.
\end{corollary}

Let $\omega$ be a palindrome with kernel $K_m$ satisfying expression (1), $1\leq i\leq f_{m+1}$.
By Property \ref{wp}, we get the relation between $P(\omega,p)$ and $P(K_m,p)$. Using Property \ref{P},
we have
$$P(\omega,p)=P(K_m,p)+f_{m+1}-i=pf_{m+1}+(\lfloor\phi p\rfloor+1)f_{m}-1+f_{m+1}-i.\eqno(2)$$

\begin{remark} As a simple application of the expression (2), we give a necessary and sufficient condition
such that the prefix $\mathbb{F}[1,n]$ is a palindrome which was proved by J.Shallit et al \cite{DMSS2014}. In fact,
$$\{n:\mathbb{F}[1,n]\in\mathcal{P}\}=\{n:\omega\in\mathcal{P},|\omega|=n,L(\omega,1)=1\}
=\{n:\omega\in\mathcal{P},|\omega|=P(\omega,1)=n\}.$$
Since $P(\omega,1)=2f_{m+1}+f_m-i-1$ and $|\omega|=f_{m+3}-2i$, we have $|\omega|=P(\omega,1)$ iff $i=f_{m+3}-2f_{m+1}-f_m+1=1\in\{1,\cdots,f_{m+1}\}.$
This means $\mathbb{F}[1,n]$ is a palindrome iff $n=f_{m+3}-2$ for $m\geq-1$,
i.e. $n=f_m-2$ for $m\geq2$, which is Theorem 14 in \cite{DMSS2014}.
\end{remark}

Define the set
$\langle K_m,p\rangle:=\{P(\omega,p):\omega\in\mathcal{P},Ker(\omega)=K_m\},$
which is the finite subset of $\mathbb{N}$. %positive integers.
By the expression (2), we have

\begin{property}\label{K}%[The range of $P(\omega,p)$]\
For $m\geq-1$, $p\geq 1$,
$$\langle K_m,p\rangle=\{pf_{m+1}+(\lfloor\phi p\rfloor+1) f_{m}-1,\cdots,(p+1)f_{m+1}+(\lfloor\phi p\rfloor+1)f_{m}-2\}.\eqno(3)$$
\end{property}

An immediately corollary is $|\langle K_m,p\rangle|=|\{1\leq i\leq f_{m+1}\}|=f_{m+1}$.

Especially,
$\langle K_m,1\rangle=\{f_{m+2}-1,\cdots,f_{m+3}-2\}.$
Thus for $m\geq -1$, two integer sets $\langle K_m,1\rangle$ and $\langle K_{m+1},1\rangle$ are consecutive.
Therefore we get a chain $\{\langle K_m,1\rangle\}_{m\geq-1}$ satisfying $\bigsqcup_{m=-1}^\infty\langle K_m,1\rangle=\mathbb{N}$,
see Tab.2 below which we call the chain structure of $\cal P$.

%of which the union is disjoint and is exactly the set $\mathbb N$

\vspace{0.2cm}

\centerline{Tab.2: The chain structure of palindromes.}
\vspace{-0.2cm}
\small
\begin{center}
\begin{tabular}{l|l|l|l|l|l|l|l|l}%
\hline
$\langle a,1\rangle$&$\langle b,1\rangle$&$\langle aa,1\rangle$&$\langle K_2,1\rangle$&$\langle K_3,1\rangle$&$\langle K_4,1\rangle$&$\langle K_5,1\rangle$&$\langle K_6,1\rangle$&$\ldots$\\
$\{1\}$&$\{2,3\}$&$\{4,5,6\}$&$\{7,\cdots,11\}$&$\{12,\cdots,19\}$
&$\{20,\cdots,32\}$&$\{33,\cdots,53\}$&$\{54,\cdots,87\}$&\\\hline
\end{tabular}
\end{center}
\normalsize

\begin{theorem}
For $n\geq1$, the number of distinct palindrome occurrences in $\mathbb{F}[1,n]$ is $n$.
\end{theorem}

\begin{proof} By the chain structure of palindromes, for $n\geq1$
\begin{equation*}
\begin{split}
&\bigsqcup_{m=-1}^\infty\langle K_m,1\rangle=\mathbb{N}\Rightarrow
|\{\omega:\omega\in\mathcal{P},P(\omega,1)=n\}|=1\Rightarrow
|\{\omega:\omega\in\mathcal{P},\omega_1\triangleright\mathbb{F}[1,n]\}|=1\\
\Rightarrow&|\{\omega:\omega\in\mathcal{P},\omega\prec\mathbb{F}[1,n],
\omega\not\!\prec\mathbb{F}[1,n-1]\}|=1.
\end{split}
\end{equation*}
This is equivalent to our theorem.
\end{proof}

\section{The recursive structure of palindromes}

In this section, we establish a recursive structure of palindromes. Using it, we will count
the number of repeated palindrome occurrences in $\mathbb{F}[1,n]$ in Section 6.

\begin{lemma}\label{L1}
$\lfloor\phi (p+\lfloor\phi p\rfloor)\rfloor=p-1$.
\end{lemma}

\begin{proof}
Notice that $Ker(aba)=b$, by Property \ref{wp}, we have $(aba)_p=a(b)_pa$. So
$$P(aba,p)=P(b,p)+1=2p+\lfloor\phi p\rfloor+1.$$

On the other hand, by Property \ref{G}, the $p$-th occurrence of $aba=r_1(a)a$ is equal to the $L(a,p)$-th occurrence of $a$. This means $L(aba,p)=L(a,L(a,p))$, i.e.,
$$P(aba,p)=P(a,P(a,p))+2=P(a,p)+\lfloor\phi P(a,p)\rfloor+2=p+\lfloor\phi p\rfloor+\lfloor\phi (p+\lfloor\phi p\rfloor)\rfloor+2.$$

Compare the two expressions of $P(aba,p)$, we have $p-1=\lfloor\phi (p+\lfloor\phi p\rfloor)\rfloor$.
\end{proof}

\begin{remark}\label{L2}
By an analogous argument, we can get interesting identity:
$p+\lfloor\phi p\rfloor=\lfloor\phi (2p+\lfloor\phi p\rfloor)\rfloor.$
\end{remark}

\begin{lemma}\label{L3} $\lfloor\phi(p+\lfloor\phi p\rfloor+1)\rfloor=p$.\end{lemma}

\begin{proof}
We only need to prove $p\leq\phi(p+\lfloor\phi p\rfloor+1)<p+1$.

By Lemma \ref{L1}, $\lfloor\phi (p+\lfloor\phi p\rfloor)\rfloor=p-1$. So
$\phi(p+\lfloor\phi p\rfloor)<p$, $\phi(p+\lfloor\phi p\rfloor+1)<p+\phi<p+1$.

On the other hand, by Corollary \ref{P1}, $P(a,p)=p+\lfloor\phi p\rfloor$,
A known result is $L(a,p)=\lfloor p/\phi\rfloor$.
Since $L(a,p)=P(a,p)$, $\lfloor p/\phi\rfloor=p+\lfloor\phi p\rfloor$.
So $\phi(\lfloor p/\phi\rfloor+1)=\phi(p+\lfloor\phi p\rfloor+1)$.
By the definition of function $\lfloor\cdot\rfloor$, $\lfloor p/\phi\rfloor+1\geq p/\phi$, i.e. $\phi(\lfloor p/\phi\rfloor+1)\geq p$.
Thus $\phi(p+\lfloor\phi p\rfloor+1)\geq p$.
\end{proof}

\begin{lemma}\label{L4}
$\lfloor\phi(2p+\lfloor\phi p\rfloor+1)\rfloor=p+\lfloor\phi p\rfloor$.
\end{lemma}

\begin{proof} By Property \ref{G}, the $p$-th occurrence of $aa=r_2(a)a$ is equal to the $L(b,p)$-th occurrence of
$a$. An immediate corollary is: the $(L(b,p)+1)$-th occurrence of
$a$ is at position $L(aa,p)+1=P(aa,p)$.
This means $P(aa,p)=P(a,P(b,p)+1)=P(a,2p+\lfloor\phi p\rfloor+1)$. By Corollary \ref{P1},
$$3p+2\lfloor\phi p\rfloor+1=2p+\lfloor\phi p\rfloor+1+\lfloor\phi(2p+\lfloor\phi p\rfloor+1)\rfloor.$$
So $p+\lfloor\phi p\rfloor=\lfloor\phi(2p+\lfloor\phi p\rfloor+1)\rfloor$.
\end{proof}

\begin{property}\label{P4.1}
$\langle K_m,p\rangle=\langle K_{m-2},P(b,p)+1\rangle\sqcup\langle K_{m-1},P(a,p)+1\rangle$ for $m\geq1$.
\end{property}

\begin{proof} By Lemma \ref{L4}, $\lfloor\phi(2p+\lfloor\phi p\rfloor+1)\rfloor=p+\lfloor\phi p\rfloor$. By Property \ref{P},
\begin{equation*}
\begin{split}
&\min\langle K_{m-2},P(b,p)+1\rangle=\min\langle K_{m-2},2p+\lfloor\phi p\rfloor+1\rangle\\
=&(2p+\lfloor\phi p\rfloor+1)f_{m-1}+(\lfloor\phi(2p+\lfloor\phi p\rfloor+1)\rfloor+1) f_{m-2}-1\\
=&(2p+\lfloor\phi p\rfloor+1)f_{m-1}+(p+\lfloor\phi p\rfloor+1) f_{m-2}-1
=pf_{m+1}+(\lfloor\phi p\rfloor+1)f_{m}-1.
\end{split}
\end{equation*}
Since $\min\langle K_m,p\rangle=pf_{m+1}+(\lfloor\phi p\rfloor+1)f_{m}-1$, $\min\langle K_{m-2},P(b,p)+1\rangle=\min\langle K_m,p\rangle$.

By Lemma \ref{L3} and Property \ref{P}, we have 
$\max\langle K_{m-1},P(a,p)+1\rangle=\max\langle K_m,p\rangle$.

Similarly, $\max\langle K_{m-2},P(b,p)+1\rangle+1=\min\langle K_{m-1},P(a,p)+1\rangle$. So the conclusion holds.
\end{proof}

For instance, taking $m=4$ and $p=1$, we have $P(b,p)+1=3$ and $P(a,p)+1=2$. Thus $\langle K_4,1\rangle=\{20,\cdots,32\}$ is the disjoint unite of $\langle K_2,3\rangle=\{20,\cdots,24\}$ and $\langle K_3,2\rangle=\{25,\cdots,32\}$.

\begin{property}[]\label{P4.2}
$\max\langle b,p\rangle=\langle a,P(a,p)+1\rangle$  for $p\geq1$.
\end{property}

\begin{proof} By Corollary \ref{P1}, $P(a,p)=p+\lfloor\phi p\rfloor$. By Lemma \ref{L3}, $\lfloor\phi(p+\lfloor\phi p\rfloor+1)\rfloor=p$. So
$$\langle a,P(a,p)+1\rangle
=\langle a,p+\lfloor\phi p\rfloor+1\rangle
=p+\lfloor\phi p\rfloor+1+\lfloor\phi(p+\lfloor\phi p\rfloor+1)\rfloor
=p+\lfloor\phi p\rfloor+1+p.$$
By expression (3), $\max\langle b,p\rangle
=(p+1)f_{1}+(\lfloor\phi p\rfloor+1)f_{0}-2=2p+\lfloor\phi p\rfloor+1=\langle a,P(a,p)+1\rangle$.
\end{proof}

From properties \ref{P4.1} and \ref{P4.2}, we can establish the following recursive relations for any $\langle K_m,p\rangle$, which we call the recursive structure of $\cal P$.
\begin{equation*}
\begin{cases}
\tau_1\langle K_m,p\rangle:=\langle K_{m-2},P(b,p)+1\rangle\sqcup\langle K_{m-1},P(a,p)+1\rangle&\text{for }m\geq1;\\
\tau_2\langle K_0,p\rangle=\tau_2\langle b,p\rangle:=\langle K_{-1},P(a,p)+1\rangle.
\end{cases}
\end{equation*}

On the other hand, (a) For any $m\geq-1$, each $\langle K_{m},1\rangle$ belongs to the recursive structure.
(b) Since $\mathbb{F}$ over alphabet $\{a,b\}$, $\mathbb{N}=\{1\}\sqcup\{P(a,\hat{p})+1\}\sqcup\{P(b,\hat{p})+1\}$. So for any $p\geq2$, there exists
$\hat{p}$ such that $P(a,\hat{p})+1=p$ or $P(b,\hat{p})+1=p$.
\begin{equation*}
\begin{cases}
\text{If there exists }\hat{p}\text{ such that }P(a,\hat{p})+1=p,
\begin{cases}
\langle K_{m},p\rangle\subset\tau_1\langle K_{m+1},\hat{p}\rangle\text{ for }m\geq0\\
\langle K_{m},p\rangle=\langle a,p\rangle\subset\tau_2\langle b,\hat{p}\rangle\text{ for }m=-1
\end{cases}\\
\text{If there exists }\hat{p}\text{ such that }P(b,\hat{p})+1=p,
\langle K_{m},p\rangle\subset\tau_1\langle K_{m+2},\hat{p}\rangle.
\end{cases}
\end{equation*}

Thus the recursive structure contains all $\langle K_m,p\rangle$, i.e. contains all palindromes in $\mathbb{F}$.
%Thus we establish the recursive structure of palindromes in $\mathbb{F}$.
\vspace{-0.3cm}
\small
\setlength{\unitlength}{1.2mm}
\begin{center}
\begin{picture}(70,65)
\put(0,0.5){\line(1,0){70}}
\put(0,63.5){\line(1,0){70}}
\put(0,0.5){\line(0,1){63}}
\put(11,0.5){\line(0,1){63}}
\put(22,0.5){\line(0,1){63}}
\put(33,0.5){\line(0,1){63}}
\put(44,0.5){\line(0,1){63}}
\put(56,0.5){\line(0,1){63}}
\put(70,0.5){\line(0,1){63}}
\put(4,1){32}
\put(4,4){31}
\put(4,10){30}
\put(4,16){29}
\put(4,19){28}
\put(4,25){27}
\put(4,28){26}
\put(4,34){25}
\put(4,40){24}
\put(4,43){23}
\put(4,49){22}
\put(4,55){21}
\put(4,58){20}
\put(1,61){$\langle K_4,1\rangle$}
\put(26,40){24}
\put(26,43){23}
\put(26,49){22}
\put(26,55){21}
\put(26,58){20}
\put(23,61){$\langle K_2,3\rangle$}
\put(15,1){32}
\put(15,4){31}
\put(15,10){30}
\put(15,16){29}
\put(15,19){28}
\put(15,25){27}
\put(15,28){26}
\put(15,34){25}
\put(12,37){$\langle K_3,2\rangle$}
\put(11,39.5){\line(1,0){59}}
\put(26,1){32}
\put(26,4){31}
\put(26,10){30}
\put(26,16){29}
\put(26,19){28}
\put(23,22){$\langle K_2,4\rangle$}
\put(22,24.5){\line(1,0){48}}
\put(37,40){24}
\put(37,43){23}
\put(37,49){22}
\put(34,52){$\langle K_1,5\rangle$}
\put(33,54.5){\line(1,0){37}}
\put(37,25){27}
\put(37,28){26}
\put(37,34){25}
\put(34,37){$\langle K_1,6\rangle$}
\put(37,1){32}
\put(37,4){31}
\put(37,10){30}
\put(34,13){$\langle K_1,7\rangle$}
\put(33,15.5){\line(1,0){37}}
\put(48,55){21}
\put(48,58){20}
\put(45,61){$\langle K_0,8\rangle$}
\put(48,40){24}
\put(48,43){23}
\put(45,46){$\langle K_0,9\rangle$}
\put(44,48.5){\line(1,0){26}}
\put(48,25){27}
\put(48,28){26}
\put(45,31){$\langle K_0,10\rangle$}
\put(44,33.5){\line(1,0){26}}
\put(48,16){29}
\put(48,19){28}
\put(45,22){$\langle K_0,11\rangle$}
\put(48,1){32}
\put(48,4){31}
\put(45,7){$\langle K_0,12\rangle$}
\put(44,9.5){\line(1,0){26}}
\put(62,55){21}
\put(57,61){$\langle K_{-1},13\rangle$}
\put(62,49){22}
\put(57,52){$\langle K_{-1},14\rangle$}
\put(62,40){24}
\put(57,46){$\langle K_{-1},15\rangle$}
\put(62,34){25}
\put(57,37){$\langle K_{-1},16\rangle$}
\put(62,25){27}
\put(57,31){$\langle K_{-1},17\rangle$}
\put(62,16){29}
\put(57,22){$\langle K_{-1},18\rangle$}
\put(62,10){30}
\put(57,13){$\langle K_{-1},19\rangle$}
\put(62,1){32}
\put(57,7){$\langle K_{-1},20\rangle$}
\end{picture}
\end{center}
\normalsize
\vspace{-0.4cm}
\centerline{(a)}

\small
\setlength{\unitlength}{1.85mm}
\newsavebox{\bb}
\savebox{\bb}(7,3.5){
\begin{picture}(7,3.5)
\put(0,0){\line(1,0){7}}
\put(0,3.5){\line(1,0){7}}
\put(0,0){\line(0,1){3.5}}
\put(7,0){\line(0,1){3.5}}
\end{picture}}
\newsavebox{\bbb}
\savebox{\bbb}(9,3.5){
\begin{picture}(9,3.5)
\put(0,0){\line(1,0){9}}
\put(0,3.5){\line(1,0){9}}
\put(0,0){\line(0,1){3.5}}
\put(9,0){\line(0,1){3.5}}
\end{picture}}

\begin{center}
\begin{picture}(60,42)
\put(1,41){$\langle K_4,1\rangle$}
\put(0,40){\usebox{\bb}}
\put(8,42){\vector(1,0){11.5}}
\put(8,42){\vector(1,-3){4.3}}
\put(21,41){$\langle K_2,3\rangle$}
\put(20,40){\usebox{\bb}}
\put(28,42){\vector(1,0){12}}
\put(28,42){\vector(1,-1){3}}
\put(41,41){$\langle K_0,8\rangle$}
\put(40,40){\usebox{\bb}}
\put(51,41){$\langle K_{-1},13\rangle$}
\put(50,40){\usebox{\bbb}}
\put(48,42){\vector(1,0){2}}
\put(31,36){$\langle K_1,5\rangle$}
\put(30,35){\usebox{\bb}}
\put(38,37){\vector(1,0){12}}
\put(38,37){\vector(1,-1){3}}
\put(51,36){$\langle K_{-1},14\rangle$}
\put(50,35){\usebox{\bbb}}
\put(41,31){$\langle K_0,9\rangle$}
\put(40,30){\usebox{\bb}}
\put(51,31){$\langle K_{-1},15\rangle$}
\put(50,30){\usebox{\bbb}}
\put(48,32){\vector(1,0){2}}
\put(11,26){$\langle K_3,2\rangle$}
\put(10,25){\usebox{\bb}}
\put(18,27){\vector(1,0){12}}
\put(18,27){\vector(2,-3){5}}
\put(31,26){$\langle K_1,6\rangle$}
\put(30,25){\usebox{\bb}}
\put(38,27){\vector(1,0){12}}
\put(38,27){\vector(1,-1){3}}
\put(51,26){$\langle K_{-1},16\rangle$}
\put(50,25){\usebox{\bbb}}
\put(40.5,21){$\langle K_0,10\rangle$}
\put(40,20){\usebox{\bb}}
\put(51,21){$\langle K_{-1},17\rangle$}
\put(50,20){\usebox{\bbb}}
\put(48,22){\vector(1,0){2}}
\put(21,16){$\langle K_2,4\rangle$}
\put(20,15){\usebox{\bb}}
\put(28,17){\vector(1,0){12}}
\put(28,17){\vector(1,-1){3}}
\put(40.5,16){$\langle K_0,11\rangle$}
\put(40,15){\usebox{\bb}}
\put(51,16){$\langle K_{-1},18\rangle$}
\put(50,15){\usebox{\bbb}}
\put(48,17){\vector(1,0){2}}
\put(31,11){$\langle K_1,7\rangle$}
\put(30,10){\usebox{\bb}}
\put(38,12){\vector(1,0){12}}
\put(38,12){\vector(1,-1){3}}
\put(51,11){$\langle K_{-1},19\rangle$}
\put(50,10){\usebox{\bbb}}
\put(40.5,6){$\langle K_0,12\rangle$}
\put(40,5){\usebox{\bb}}
\put(51,6){$\langle K_{-1},20\rangle$}
\put(50,5){\usebox{\bbb}}
\put(48,7){\vector(1,0){2}}
\end{picture}
\end{center}
\normalsize
\vspace{-1.2cm}
\centerline{(b)}
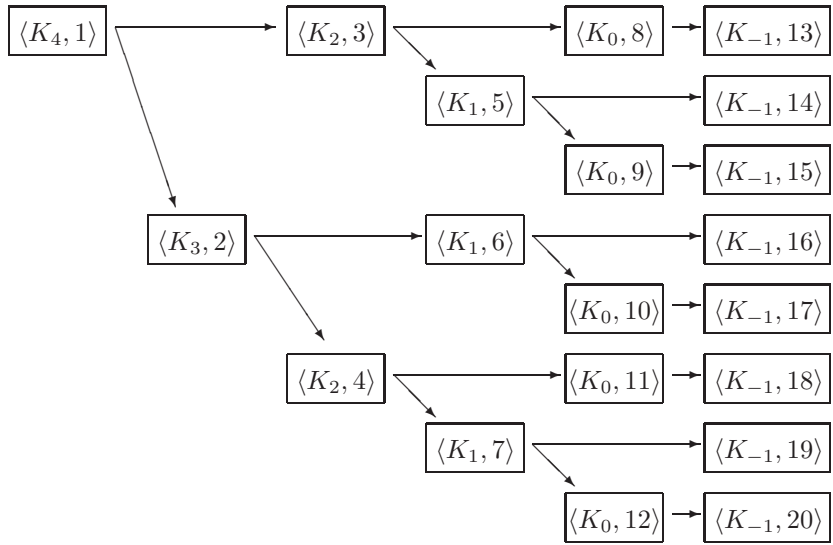
\centerline{Fig.1: The recursive structure of palindromes from $\langle K_4,1\rangle$.}

\vspace{0.2cm}

%By Property \ref{K}, the $i$-th element in $\langle K_m,p\rangle$ is $pf_{m+1}+(\lfloor\phi p\rfloor+1) f_{m}+i-2$ for $1\leq i\leq f_{m+1}$.
By the recursive structure, the property below give the relation between the number of palindromes ending at position $pf_{m+1}+(\lfloor\phi p\rfloor+1) f_{m}+i-2$ and $f_{m+2}+i-2$, which are the $i$-th element in $\langle K_m,p\rangle$ and $\langle K_m,1\rangle$ respectively, where $1\leq i\leq f_{m+1}$.

\begin{property}[]\label{P5.5} For $1\leq i\leq f_{m+1}$,
\begin{equation*}
\begin{split}
&\{\omega:\omega\in\mathcal{P},\omega\triangleright\mathbb{F}[1,f_{m+2}+i-2]\}\\
=&\{\omega:\omega\in\mathcal{P},\omega\triangleright\mathbb{F}[1,pf_{m+1}+(\lfloor\phi p\rfloor+1) f_{m}+i-2],Ker(\omega)=K_j,-1\leq j\leq m\}.
\end{split}
\end{equation*}
\end{property}

For instance, taking $m=2$, $p=3$, $i=2$. All palindromes ending at position 8 are $\{a,aba,ababa\}$.
All palindromes ending at position 21 are $\{a,aba,ababa,\omega\}$ where $\omega=ababaababa$.
Since $Ker(a)=K_{-1}$, $Ker(aba)=K_0$, $Ker(ababa)=K_2$ and $Ker(\omega)=K_3$,
only $\{a,aba,ababa\}$ are palindromes with kernel $K_{j}$, $-1\leq j\leq 2$.

\section{The number of repeated palindrome occurrences in $\mathbb{F}[1,n]$}

By the recursive structure of palindromes in $\mathbb{F}$, we can count the number of palindromes end at position $n$, denoted by $A(n):=|\{\omega_p:\omega_p\triangleright\mathbb{F}[1,n]\}|.$
And obviously the number of repeated palindrome occurrences in $\mathbb{F}[1,n]$ is $B(n):=\sum_{i=1}^nA(i)$.

By Property \ref{P4.1}, \ref{P4.2}, \ref{P5.5}, and the recursive structure of palindromes, we have:

\begin{theorem}\label{P4.3} The vectors $[A(1)]=[1]$, $[A(2),A(3)]=[1,2]$ and for $m\geq3$
\begin{equation*}
\begin{split}
&[A(f_{m}-1),\cdots,A(f_{m+1}-2)]\\
=&[A(f_{m-2}-1),\cdots,A(f_{m-1}-2),A(f_{m-1}-1),\cdots,A(f_{m}-2)]
+[\underbrace{1,\cdots,1}_{f_{m-1}}].
\end{split}
\end{equation*}
\end{theorem}

The first few values of $A(n)$ are
$[A(1)]=[1]$,
$[A(2),A(3)]=[1,2]$,

$[A(4),A(5),A(6)]=[1,1,2]+[1,1,1]=[2,2,3]$,

$[A(7),\cdots,A(11)]=[1,2,2,2,3]+[\underbrace{1,\cdots,1}_5]=[2,3,3,3,4]$,

$[A(12),\cdots,A(19)]=[2,2,3,2,3,3,3,4]+[\underbrace{1,\cdots,1}_8]=[3,3,4,3,4,4,4,5]$.

\begin{theorem}[] The number of repeated palindrome occurrences in $\mathbb{F}[1,n]$ is $B(n)=\sum\limits_{i=1}^nA(i)$.
\end{theorem}

By considering $B(f_{m}-2)$ for $m\geq2$, we can determine the expressions of $B(f_{m})$ etc, and give a fast algorithm of $B(n)$ for all $n\geq1$.

Let $C(m)=B(f_{m+1}-2)-B(f_{m}-2)$, then an immediate corollary of Property \ref{P4.3} is
$$\begin{array}{c}
C(m)=\sum\limits_{n=f_{m}-1}^{f_{m+1}-2}A(n)
=\sum\limits_{n=f_{m-2}-1}^{f_{m-1}-2}A(n)+\sum\limits_{n=f_{m-1}-1}^{f_{m}-2}A(n)+f_{m-1}.
\end{array}$$
This means $C(m)=C(m-2)+C(m-1)+f_{m-1}$.

By induction, we can have Property 6.3 and Remark 6.4 easily.

\begin{property} $C(m)%=\sum_{i=-1}^{m-1} f_if_{m-i-2}-f_{m-1}
=\frac{m+1}{5}f_{m+1}+\frac{m-2}{5}f_{m-1}$ for $m\geq1$.
%$C(m)=\sum_{i=0}^{m-1} f_if_{m-i-2}$
\end{property}

\begin{remark}[] In fact, we get the expression of $C(m)$ by an interesting identity that
$$\begin{array}{c}
\sum\limits_{i=-1}^{m}f_if_{m-i-1}=\frac{m+2}{5}f_{m+2}+\frac{m+4}{5}f_{m}\text{ for } m\geq1,
\end{array}$$
\end{remark}

By $B(f_{m}-2)=\sum_{n=1}^{m-1}C(n)$,
$C(m)=\frac{m+1}{5}f_{m+1}+\frac{m-2}{5}f_{m-1}$ and $\sum_{i=-1}^mf_i= f_{m+2}-1$, we have

\begin{property} $B(f_{m}-2)=\frac{m-3}{5}f_{m+2}+\frac{m-1}{5}f_{m}+2$ for $m\geq2$.
\end{property}

\vspace{0.2cm}

By Theorem \ref{P4.3}, we can get the property below easily by induction.

\begin{property}\
$A(f_{m}-2)=m-1$, $A(f_{m}-1)=\lfloor\frac{m+1}{2}\rfloor$, $A(f_{m})=\lfloor\frac{m+2}{2}\rfloor$, $A(f_{m}-1)+A(f_{m})=m+1$.
\end{property}

Thus, by the two properties above, we can determine the expressions of $B(f_{m}-3)$, $B(f_{m}-1)$, $B(f_{m})$.
Especially, since $B(f_{m})=B(f_{m}-2)+A(f_{m}-1)+A(f_{m})$, we have

\begin{theorem}[]
The number of repeated palindrome occurrences in $F_m$ is
$$\begin{array}{c}
B(f_m)=\frac{m-3}{5}f_{m+2}+\frac{m-1}{5}f_{m}+m+3.
\end{array}$$
\end{theorem}

For instance, taking $m=5$, $B(13)=\frac{2}{5}f_{7}+\frac{4}{5}f_{5}+10
=\frac{2}{5}\times34+\frac{4}{5}\times13+8=32$.

\vspace{0.2cm}

For any $n\geq1$, let $m$ such that $f_m\leq n+1< f_{m+1}$. Since we already determine the expression of $B(f_m-2)$, in order to give a fast algorithm
of $B(n)$, we only need to calculate $\sum_{i=f_{m}-1}^{n}A(i)$. One method is calculating $A(n)$ by Theorem \ref{P4.3}, the other method is using the corollary as below.

\begin{corollary}\label{c5} For $n\geq1$, let $m$ such that $f_m\leq n+1< f_{m+1}$, then
\begin{equation*}
\sum_{i=f_{m}-1}^{n}A(i)=
\begin{cases}
\sum\limits_{i=f_{m-2}-1}^{n-f_{m-1}}A(i)+n-f_{m}+2,&n+1< 2f_{m-1};\\
\sum\limits_{i=f_{m-1}-1}^{n-f_{m-1}}A(i)+n+\frac{m-11}{5}f_{m-1}+\frac{m+1}{5}f_{m-3}+2,
&otherwise.
\end{cases}
\end{equation*}
\end{corollary}

\begin{proof} By Theorem \ref{P4.3},
when $f_m\leq n+1< 2f_{m-1}$,
$$\begin{array}{c}
\sum\limits_{i=f_{m}-1}^{n}A(i)=\sum\limits_{i=f_{m-2}-1}^{n-f_{m-1}}(A(i)+1)
=\sum\limits_{i=f_{m-2}-1}^{n-f_{m-1}}A(i)+n-f_{m}+2.
\end{array}$$

When $2f_{m-1}\leq n+1< f_{m+1}$, $\sum_{i=f_{m}-1}^{n}A(i)
=\sum\limits_{i=f_{m}-1}^{2f_{m-1}-2}A(i)+\sum\limits_{i=2f_{m-1}-1}^{n}A(i)$,
where
$$\begin{array}{rl}
\sum\limits_{i=f_{m}-1}^{2f_{m-1}-2}A(i)=&\sum\limits_{i=f_{m-2}-1}^{f_{m-1}-2}A(i)+ f_{m-3}=C(m-2)+f_{m-3}\\
=&\frac{m-1}{5}f_{m-1}+\frac{m-4}{5}f_{m-3}+f_{m-3}
=\frac{m-1}{5}f_{m-1}+\frac{m+1}{5}f_{m-3}.\\
\sum\limits_{i=2f_{m-1}-1}^{n}A(i)=&\sum\limits_{i=f_{m-1}-1}^{n-f_{m-1}}A(i)+n-2f_{m-1}+2.
\end{array}$$

Thus we have the expression that
$\sum\limits_{i=f_{m}-1}^{n}A(i)
=\sum\limits_{i=f_{m-1}-1}^{n-f_{m-1}}A(i)+n+\frac{m-11}{5}f_{m-1}+\frac{m+1}{5}f_{m-3}+2,$
which yields the conclusion holds.
\end{proof}

\noindent\textbf{Example.} We calculate $\sum_{i=20}^{29}A(i)$. One method is using Theorem \ref{P4.3}, we have

$[A(1)]=[1]$,
$[A(2),A(3)]=[1,2]$,
$[A(4),A(5),A(6)]=[2,2,3]$,
$[A(7),\cdots,A(11)]=[2,3,3,3,4]$,

$[A(12),\cdots,A(19)]=[3,3,4,3,4,4,4,5]$,
$[A(20),\cdots,A(32)]=[3,4,4,4,5,4,4,5,4,5,5,5,6]$.

So $\sum_{i=20}^{29}A(i)=3+4+4+4+5+4+4+5+4+5=42$.

The other method is using Corollary \ref{c5}.

Since $f_6=21\leq 29+1< f_{7}=34$, $m=6$. Moreover $2f_5=26\leq 29+1$,
$$\begin{array}{c}
\sum\limits_{i=f_{6}-1}^{29}A(i)
=\sum\limits_{i=f_{5}-1}^{29-f_{5}}A(i)+29+\frac{-5}{5}f_{5}+\frac{7}{5}f_{3}+2
=\sum\limits_{i=f_{5}-1}^{16}A(i)+25.
\end{array}$$

Similarly, $\sum\limits_{i=f_{5}-1}^{16}A(i)=\sum\limits_{i=f_{4}-1}^{8}A(i)+12$ and
$\sum\limits_{i=f_{4}-1}^{8}A(i)=A(2)+A(3)+2=5$.

Thus $\sum_{i=f_{6}-1}^{29}A(i)=25+12+5=42$.

\begin{algorithm}[The number of repeated palindrome occurrences, $B(n)$]\

Step 1. Find the $m$ such that $f_m\leq n+1< f_{m+1}$, then $B(n)=B(f_m-2)+\sum_{i=f_{m}-1}^{n}A(i)$.

Step 2. Calculate $B(f_m-2)$ by the expression in Property 6.5;

Calculate $\sum_{i=f_{m}-1}^{n}A(i)$ by the recursive relation in Corollary \ref{c5} or Theorem \ref{P4.3}.
\end{algorithm}

For instance, since $f_6\leq 29+1< f_{7}$, we have $m=6$.
By Theorem 6.2,
$B(29)=B(19)+\sum_{i=20}^{29}A(i)$.
By Property 6.5, $B(19)=%B(f_6-2)=
\frac{3}{5}f_{8}+\frac{5}{5}f_{6}+2=56$.
By Theorem \ref{P4.3} or Corollary \ref{c5}, $\sum_{i=20}^{29}A(i)=42$.
Thus $B(29)=98$.

\vspace{0.5cm}

\noindent\textbf{\Large{Acknowledgments}}

\vspace{0.4cm}

The research is supported by the Grant NSF No.11431007, No.11271223 and No.11371210.

\end{CJK*}
\end{document}